\documentclass[a4paper,10pt,intlimits,ngerman]{amsart}
\usepackage[varg]{txfonts}
\usepackage[unicode]{hyperref}
\usepackage{graphicx}
\usepackage{color}

\newtheorem{theorem}{Theorem}[section]

\newtheorem{lemma}[theorem]{Lemma}

\newtheorem{claim}[theorem]{Claim}
\hypersetup{
breaklinks=true,
colorlinks=true,
linkcolor=blue,
citecolor=blue,
urlcolor=blue,
}

\numberwithin{equation}{section}

\DeclareMathOperator{\diam}{diam}

\title[]{Curves between Lipschitz and $C^1$ and their relation to geometric knot theory}
\author{
Simon Blatt}
\address[Simon Blatt]{Paris Lodron Universit\"at Salzburg, Hellbrunner Strasse 34, 5020 Salzburg, Austria}
\email{simon.blatt@sbg.ac.at}
\date{\today}

\begin{document}

%\linenumbers

\begin{abstract}
 In this article we investigate regular curves whose derivatives have vanishing mean oscillations. We show that smoothing these curves using a standard mollifier one gets regular curves again.
 
 We apply this result to solve a couple of open problems. We show that curves with finite M\"obius energy can be approximated by smooth curves in the energy space $W^{\frac 32,2}$ such that the energy converges which answers a question of He. Furthermore, we extend the result of Scholtes on the $\Gamma$-convergence of the discrete M\"obius energies towards the M\"obius energy and prove conjectures of Ishizeki and Nagasawa on certain parts of a decomposition of the M\"obius energy. Finally, we extend a theorem of Wu on inscribed polygons to curves with derivatives with vanishing mean oscillation.
\end{abstract}

\maketitle

\tableofcontents

\section{Introduction}

Approximating functions by functions with better regularity properties was, is, and will certainly remain to be one of the most important techniques in analysis. In this short note we want to contribute to this topic. We consider regular closed curves with regularity somewhere between $C^1$ and merely Lipschitz continuity. One ends up looking at such curves, if one assumes that the curve is parametrized by arc-length and lies in some critical fractional Sobolev space $W^{1+s,\frac 1 s}$, $s \in (0,1)$ - which is known not to embed into $C^1$. But still the fact that the curve is of class $W^{1+s, \frac 1 s} $ gives us some subtle new information on the derivative that we will use in this article. For example, the derivative of the curve $\gamma:\mathbb R / \mathbb Z \rightarrow \mathbb R^n$ then belongs to the space $VMO(\mathbb R / \mathbb Z, \mathbb R^n)$ of all functions with vanishing mean oscillation, i.e.
\begin{equation*}
 \lim_{r\rightarrow 0} \sup_{x \in \mathbb R / \mathbb Z} \frac 1 {2r} \int_{B_r(x)} |\gamma'(y) - \overline{\gamma'}_{B_r(x)}| dy =0.
\end{equation*}
Here, $\overline{\gamma'}_{B_r(x)}$ denotes the integral mean of the function $\gamma'$ over the ball $B_r(x).$
Let $\eta \in C^\infty(\mathbb R, [0, \infty))$ be such that $\eta \equiv 0$ on $\mathbb R \setminus (-1,1)$ and $\int_{\mathbb R} \eta(x) dx =1$. For $\varepsilon >0$ we consider the smoothing kernels 
$$\eta_\varepsilon (x) = \frac 1 \varepsilon \eta ( \frac x \varepsilon).$$ 
Though for merely regular curves $\gamma \in C^{0,1}(\mathbb R / \mathbb Z, \mathbb R^n)$ we cannot expect that the smoothened functions $\gamma_\varepsilon$ are regular curves, the situation changes drastically, if we assume that $\gamma'$ has vanishing mean oscillation.
We will start with proving the following surprising theorem:

\begin{theorem} \label{thm:ApproximatingVMOCurves}
 Let $\gamma \in C^{0,1}(\mathbb R / \mathbb Z, \mathbb R^n)$ be a curve parametrized by arc-length with $\gamma'\in VMO (\mathbb R / \mathbb Z, \mathbb R^n)$. For $\varepsilon >0$ we consider the smoothened functions $\gamma_\varepsilon = \gamma \ast \eta_\varepsilon$. Then the absolute value of the derivative $|\gamma_\varepsilon'|$ converges uniformly to $|\gamma'| = 1$. So especially, the curves $\gamma_\varepsilon$ are regular if $\varepsilon$ is small enough.
\end{theorem}

Sometimes one might need that also the approximating curves are parametrized by arc-length and have the same length as the original curve. In this case the following theorem can help. For denote the length of a curve $\gamma$ by $L(\gamma)$.

\begin{theorem} \label{thm:ApproximationArcLength}
 Let $\tilde \gamma_\varepsilon: \mathbb R / \mathbb Z \rightarrow \mathbb R^n$ be the re-parametrization by arc-length of the unit length curve $\frac 1 {L(\gamma_\varepsilon)} \gamma_\varepsilon$ that satisfies $\tilde \gamma_{\varepsilon}(0) = \frac 1 {L(\gamma_\varepsilon)} \gamma_\varepsilon(0)$. Then $\tilde \gamma_\varepsilon$ still converges to the curve $\gamma$ in $W^{\frac 32, 2}$.
\end{theorem}

Though the proof of the Theorem~\ref{thm:ApproximatingVMOCurves} is extraordinarily elementary and short, it is the impression of the author that this result and the techniques that lead to it are unknown to the community. In the last section, we will show how to apply the techniques of this article in order to answer some open questions in the literature and settle some conjectures in the context of knot energies. All the statements of the theorems are known for curves that possess more regularity than we can naturally assume. The approximation techniques above allow to extend these statements to curves of bounded M\"obius energy -- which is the most natural assumption for these theorems. Let us just pick out one particular open question due to Zheng-Xu He.

Jun O'Hara introduced the M\"obius energy \cite{OHara1991}
\begin{equation*}
   E_{\text{m\"ob}}(\gamma):= \iint_{\mathbb R / \mathbb Z} \left(\frac 1 {|\gamma(x)- \gamma(y)|^2} - \frac 1 {d_\gamma(x,y)^2} \right) |\gamma'(x)||\gamma'(y)| dx dy 
 \end{equation*}
for regular curves $\gamma \in C^{0,1}(\mathbb R / \mathbb Z, \mathbb R^n)$ which was the first geometric implementation of the concept of knot energy. In the influential paper \cite{Freedman1994}, Freedman, He, and Wang discussed many interesting properties of this energy including its invariance under M\"obius transformations.

In his article \cite{He2000}, Zheng-Xu He asked whether any regular curve of bounded M\"obius energy can be approximated by smooth curves such that the energy converges.
We will use the above approximation result together with the characterization of curves of finite M\"obius energy in \cite{Blatt2012} to give the following answer:

\begin{theorem} \label{thm:ApproximatingBoundedEnergyCurves}
 Let $\gamma \in C^{0,1}(\mathbb R / \mathbb Z, \mathbb R^n)$ be a curve parametrized by arc-length such that O'Hara's M\"obius energy $ E_{\text{m\"ob}}(\gamma)$
 is bounded. Then there is a constant $\varepsilon_0>0$ such that $\gamma_\varepsilon$ are smooth regular curves for all $0 < \varepsilon < \varepsilon_0$ converging to $\gamma$ in the fractional Sobolev space $W^{\frac 32,2}$ and in energy, i.e. $E^2 (\gamma_\varepsilon ) \rightarrow E^2 (\gamma)$ for $\varepsilon  \rightarrow 0$.
\end{theorem}

We hope that the list of applications, although far from being complete, convinces the reader that the results and techniques developed in this article are of great importance for the analysis of critical knot energies for curves.

\section{Approximation by Smooth Curves - Proof of Theorem~\ref{thm:ApproximatingVMOCurves} and Theorem~\ref{thm:ApproximationArcLength}} \label{sec:Approximation} 

\begin{proof}[Proof of Theorem~\ref{thm:ApproximatingVMOCurves}]
Setting
 $$
  a_r := \frac 1 {2r}\int_{B_r(x)} \gamma' dy
 $$
for $r \leq \frac 1 2$ we observe that
 \begin{align*}
  \frac 1 {2r}\int_{B_r (x)} |\gamma' - a_r| dy = o(1)
 \end{align*}
as $r \rightarrow 0$ since $\gamma'$ has vanishing mean oscillation.

We calculate using the triangle inequality and the estimate above 
 \begin{align} \label{eq:convergenceOfMean}
  |a_r| = \frac 1 {2r} \int_{B_r(x)} |a_r| dy \geq \frac 1 {2r} \int \left( |\gamma'| - |  |a_r| -  |\gamma'|| \right) \geq 1- \int_{B_r(x)} |a_r - \gamma' | dx \geq 1 - o(1)
 \end{align}
as $r \rightarrow 0$.

So we derive
 \begin{align*}
  |\gamma'_\varepsilon (x) | & = |\int_{B_\varepsilon(0)} \gamma(x+y) \eta_\varepsilon (y) dy |
  \geq | a_{\varepsilon}|
  - \int_{B_{\varepsilon(0)}} |\gamma' (x-y)- a_\varepsilon| |\eta_\varepsilon(y)| dy
 \\ &\geq |a_\varepsilon| - C o(1) \geq 1 - o(1),
 \end{align*}
where we used the definition of $VMO$ and \eqref{eq:convergenceOfMean}.
Hence,
$|\gamma_\varepsilon'| \rightarrow |\gamma'|=1$ uniformly and especially $\gamma$ is a regular curve for $\varepsilon >0$ small enough.
This completes the proof of Theorem~\ref{thm:ApproximatingVMOCurves}.
\end{proof}
%Kurven sollten auf 

\begin{proof}[Proof of Theorem~\ref{thm:ApproximationArcLength}]
Let us now consider the curves $\tilde \gamma_\varepsilon$ which apparently converge to $\gamma$ uniformly and hence especially in $L^2$. We now show that the derivatives of these curves satisfy
\begin{equation}
 \lim_{\varepsilon \rightarrow 0} \lfloor \tilde \gamma'_\varepsilon - \gamma' \rfloor_{W^{\frac 1 2,2} } =0
\end{equation}
using Vitali's theorem where
\begin{equation*}
 \lfloor f \rfloor_{W^{\frac 1 2,2} } = \int_{\mathbb R / \mathbb Z} \int_{\mathbb R / \mathbb Z} \frac {|f(x) - f(y)|^2}{|x-y|^2} dx dy
\end{equation*}
denotes the Douglas functional also known as Gagliardo semi-norm.
We therefore consider the integrand
\begin{equation*}
 I_\varepsilon (x,y) := \frac {|(\tilde \gamma'_\varepsilon(x) - \gamma'(x) ) - (\tilde \gamma_\varepsilon (y) - \gamma'(y))|^2}{|x-y|^2}
\end{equation*}
which converges pointwise almost everywhere to $0$ and can be estimate from above by
\begin{equation*}
   \frac {|\tilde \gamma'_\varepsilon(x) - \tilde \gamma_\varepsilon'(y)|^2}{|x-y|^2} + \frac {| \gamma'(x) - \gamma'(y)|^2}{|x-y|^2}
\end{equation*}
Let us now consider the bi-Lipschitz transformation 
\begin{gather*} \psi: (\mathbb R / \mathbb Z)^2 \rightarrow  (\mathbb R / \mathbb Z)^2 \\
(x,y) \mapsto (s(x),s(y)). 
\end{gather*}
For $E \subset (\mathbb R / \mathbb Z)^2$
we have
\begin{align*}
 \iint_{E} \frac {|\tilde \gamma'_\varepsilon(x) - \tilde \gamma_\varepsilon'(y)|^2}{|x-y|^2} dx dy & \leq C \iint_{\psi^{-1}(E)} \frac {| \gamma'_\varepsilon(x) -  \gamma_\varepsilon'(y)|^2}{|x-y|^2} |\gamma_\varepsilon'(x)| |\gamma_\varepsilon'(y) |dx dy \\
 & \leq C \iint_{\psi^{-1}(E)} \frac {| \gamma'_\varepsilon(x) -  \gamma'_\varepsilon(y)|^2}{|x-y|^2}dx dy.
\end{align*}
Since the integrands 
\begin{equation*}
  \frac {| \gamma'_\varepsilon(x) -  \gamma_\varepsilon'(y)|^2}{|x-y|^2}
\end{equation*}
are uniformly integrable, for every $\varepsilon_0>0$ there is an $\delta >0$ such that
$|\psi^{-}(E)|\leq \delta$ implies
\begin{equation*}
 \iint_{E} \frac {|\tilde \gamma'_\varepsilon(x) - \tilde \gamma_\varepsilon'(y)|^2}{|x-y|^2} dx dy \leq \varepsilon.
\end{equation*}
But, as $\psi^{-1}$ is a Lipschitz mapping, we get that there is an $\tilde \delta >0$ such that $|E| \leq \tilde \delta$ implies $|\psi^{-1}(E)|$ and hence
\begin{equation*}
 \iint_{E} \frac {|\tilde \gamma'_\varepsilon(x) - \tilde \gamma_\varepsilon'(y)|^2}{|x-y|^2} dx dy \leq \varepsilon.
\end{equation*}
 
Thus, $ I_\varepsilon$ is uniformly integrable and we can apply Vitali's theorem to prove the claim.

\end{proof}

\section{Applications}

We want to present several applications of Theorem~\ref{thm:ApproximatingVMOCurves}. We will start with analyzing the convergence of the M\"obius energy and the parts of its decomposition found by Ishikezi and Nagasawa if the original curve has bounded M\"obius energy. Unfortunately, the smoothened curves $\gamma_\varepsilon$ in general do not converge in $W^{1, \infty}$ -- so we cannot apply the fact that the M\"obius energy is $C^1$ in $W^{\frac 32,2} \cap W^{1,\infty}$ \cite[Theorem~II]{Blatt2015a}. We will show how to use the convergence of $|\gamma_\varepsilon'|$ from Theorem~\ref{thm:ApproximatingVMOCurves} together with bi-Lipschitz estimates in order to prove convergence in energy.

\subsection{Fractional Sobolev Spaces}

For the rest of the article we need the classification of curves of finite energy $E^\alpha$ in \cite{Blatt2012} using fractional Sobolev spaces. For $s \in (0,1)$, $p \in [1,\infty)$ and $k\in \mathbb N_0$ the space $W^{k+s,p}(\mathbb R / \mathbb Z, \mathbb R^n)$ consists of all functions $f\in W^{k,p}(\mathbb R / \mathbb R, \mathbb R^n )$ for which
$$
 |f^{(k)}|_{W^{s,p}}:= \left( \int_{\mathbb R / \mathbb Z} \int_{\mathbb R / \mathbb Z} \frac {|f^{(k)}(x) - f^{(k)}(y)|^p}{|x-y|^{1+sp}} dx dy \right)^{\frac 1 p} 
$$
is finite. This space is equipped with the norm $\|f\|_{W^{k+s,p}} := \|f\|_{W^{k,p}} + |f^{(k)}|_{W^{s,p}}.$  For an thorough discussion of the subject of fractional Sobolev space we point the reader to the monograph of Triebel \cite{Triebel1983}. Chapter 7 of \cite{Adams2003} and the very nicely written and easy accessible introduction to the subject \cite{DiNezza2012}.

The following result is a special case of Theorem~1.1 in \cite{Blatt2012}:

\begin{theorem} [Classification of curves with finite M\"obius energy] \label{thm:EnergySpace}
Let $\gamma \in C^{0,1}(\mathbb R / \mathbb Z, \mathbb R^n)$ be a curve parametrized by arc-length. Then the M\"obius energy $E_{\text{m\"ob}}(\gamma)$ is finite if and only if $\gamma$ is bi-Lipschitz and belongs to $W^{\frac 3 2,2}(\mathbb R / \mathbb Z, \mathbb R^n)$. 
\end{theorem}

In the following, we will use the well-known fact, that $f \in W^{s,p} (\mathbb R / \mathbb Z, \mathbb R^n)$, $s \in (0,1)$, $p = \frac 1 s$, implies $f \in VMO$. This follows for example from the line of inequalities 
\begin{align*}
 \frac 1 {2r}\int_{B_r(x) } |f - \overline f_{B_r(x)}| dx & \leq \frac 1 {(2r)^2} \int_{B_r(x)} \int_{B_r(x)} |f(y) - f(z)| dy dz \\ 
 &\leq \frac {1} {(2r)^{2/p}} \left(\int_{B_r(x)} \int_{B_r(x)} |f(y) - f(z)|^p dy dz\right)^{\frac 1 p} =  \left(\int_{B_r(x)} \int_{B_r(x)} \frac {|f(y) - f(z)|^p }{{2r}^2} dy dz\right)^{\frac 1 p} \\ 
 & \leq \left(\int_{B_r(x)} \int_{B_{r}(x)} \frac {|f(y) - f(z)| }{|y-z|^p} dy dz\right)^{\frac 1 p} \\ 
 & \leq \left(\int_{y \in \mathbb R / \mathbb Z} \int_{B_{2r}(0)} \frac {|f(z+w) - f(z)|^p }{|w|^{1+sp}} dw dz\right)^{\frac 1 p}  \rightarrow 0
\end{align*}
for $r \rightarrow 0$.

Applying this to $f= \gamma'$, in view of Theorem~\ref{thm:EnergySpace} the velocity of a curve parametrized by arc-length of finite M\"obius energy belongs to $VMO$. Hence, we can apply Theorem~\ref{thm:ApproximatingVMOCurves}.

\subsection{Convergence of Some Critical Knot Energies}

\subsubsection{The M\"obius Energy}

As a first application, we want to answer a question due to He \cite{He2000}[Question 8 in Section 7]. Zhen-Xu He asked, whether a curve of bounded M\"obius energy can be approximated by smooth curves such that the energies of these curves converge to the energy of the initial curve. Then following lemma shows that this is indeed the case and that one can just use the mollified curves $\gamma_\varepsilon$. This lemma together with Theorem~\ref{thm:ApproximatingVMOCurves} obviously proves Theorem~\ref{thm:ApproximatingBoundedEnergyCurves}.

\begin{lemma} [Convergence of the M\"obius energy]\label{lem:ConvOfMoebiusEnergy}
 Let $\gamma \in C^{0,1}(\mathbb R / \mathbb Z, \mathbb R^n)$ be parametrized by arc-length of finite M\"obius energy. Then we have
 $ E_{\text{m\"ob}} (\gamma_\varepsilon) \rightarrow E(\gamma)$.
\end{lemma}

\begin{proof} We use Vitali's convergence theorem to prove this lemma. 
Setting $I_\gamma(z,w):= \left( \frac 1 {|\gamma (x+w) - \gamma(x)|^2} - \frac 1 {d_\gamma(x,x+w)^2} \right) |\gamma'(x)| \, |\gamma' (x+w)|$, we get
$$
  E_{\text{m\"ob}}(\gamma)  = \int_{\mathbb R / \mathbb Z} \int_{-\frac 1 2} ^{\frac 12}   I_\gamma(z,w) \,  dx dw.
$$
As $|\gamma_\varepsilon'|$ converges pointwise to $| \gamma'|$ by Theorem~\ref{thm:ApproximatingVMOCurves} and $\gamma_\varepsilon$ converges to $\gamma$ pointwise, the integrand $I_{\gamma_\varepsilon}(x,w)$ also converges
to $I_{\gamma}(x,w)$ pointwise. Let us show that the integrands are uniformly integrable. For this purpose we only have to consider points close to the diagonal, i.e. we will only integrate over $x,y \in \mathbb R / \mathbb Z$ with $|x-y|\leq \frac 1 4$, since on the rest of the domain the bi-Lipschitz estimate gives us a uniform bound on the integrand.

We have for $\varepsilon >0$ small enough and $|w|\leq \frac 1 4$ that $d_{\gamma_\varepsilon}(x+w,x) = \int_0^1 |\gamma_\varepsilon'(x+sw)|ds$.
Together with the identity $\gamma_\varepsilon(x+w) - \gamma(x) = w\int_{0}^1 \gamma_\varepsilon'(x+sw)ds$ we get
 \begin{align*}
  I_\varepsilon(x,w) &=\left( \frac 1 {|\gamma_\varepsilon(x+w) - \gamma_\varepsilon(x)|^2 } - \frac 1 {d_{\gamma_\varepsilon}(x,x+w)^2} \right) \\ & = \frac {|w|^4}{|\gamma_\varepsilon(x+w) -\gamma_\varepsilon(x)| ^2 d_{\gamma_\varepsilon} (x,x+w)^2}  \\ & \quad \quad \quad \quad \quad \quad \quad \left( \frac { \int_{0}^1 \int_0^1 |\gamma_\varepsilon'(x+s_1 w)| |\gamma_\varepsilon'(x+s_2 w| - \langle \gamma_\varepsilon'(x+s_1 w) \gamma_\varepsilon'(x+s_2 w)\rangle ds_1 ds_2 }{ |w|^2}\right).
 \end{align*}
As all vectors $ a,b \in \mathbb R^n \setminus \{0\}$ satisfy
$$
 |a| |b| - \langle a,b \rangle = \frac {|a||b|}{2} \left|\frac a {|a|} - \frac b {|b|} \right|^2
$$
and
$$
   \left|\frac a {|a|} - \frac b {|b|} \right| \leq \frac {|a-b|}{|a|} + |b| \left| \frac 1 {|a|} - \frac 1 {|b|} \right|
  \leq \frac  {2 |a-b|} {|a|},
$$
we get
$$
 |a| |b| - \langle a,b \rangle \leq \frac {|b|} {|a|} |a-b|^2.
$$
Applying this inequality to $a=\gamma'(x+s_1 w)$ and $b = \gamma'(x+s_2 w)$ , we arrive at
\begin{equation} \label{eq:EstIntegrand}
 |I_{\gamma_\varepsilon}(x,w)| \leq C  \left( \frac { \int_{0}^1 \int_0^1 |\gamma'(x+s_1 w)-\gamma'(x+s_2 w)|^2ds_1ds_2}{ |w|^2}\right) := C \tilde I_{\gamma_\varepsilon} (x,w)
\end{equation}
for all $|w|\leq \frac 14$ and $\varepsilon >0$ small enough.
Let us now show that $\tilde I_{\gamma_\varepsilon}(x,w)$ converges to $\tilde I_{\gamma}(x,w)$ in $L^1(\mathbb R \times \mathbb Z \times [-\frac 12, \frac 12])$ which implies that $I_{\gamma_\varepsilon}(x,w) \leq \tilde I_{\gamma_\varepsilon}(x,w) $ is uniformly integrable. 

We calculate
\begin{align*}
  &\left|\left(\int_{\mathbb R / \mathbb Z} \int_{-\frac 1 2}^{\frac 1 2} {\frac 1 2}\frac{|\gamma_\varepsilon(x+w) - \gamma_\varepsilon (x)|^2}{w^2}dwdx \right) ^{\frac 1 2 } -  \left(\int_{\mathbb R / \mathbb Z} \int_{-\frac 1 2}^{\frac 1 2} \frac 1 2\frac{|\gamma(x+w) - \gamma (x)|^2}{w^2} dwdx \right) ^{\frac 1 2 }\right| \\
 & \leq  \left(\int_{\mathbb R / \mathbb Z} \int_{-\frac 1 2}^{\frac 1 2}  \frac{ |(\gamma_\varepsilon (x+w) - \gamma_\varepsilon (x) - (\gamma(x+w) - \gamma(x)) |^2} {w^2}dw dx  \right) ^{\frac 1 2}.
\end{align*}
Using the definition of the convolution, we can estimate this by
\begin{align*}
 & \quad \left( \int_{\mathbb R / \mathbb Z} \int_{-\frac 1 2}^{\frac 1 2}  \frac{  \int_{B_\varepsilon(0)} |\left(\gamma(x+w -y) - \gamma (x-y) - (\gamma(x+w) - \gamma(x))\right) \eta_\varepsilon(y) dy |^2} {w^2}dw dx  \right)^{\frac 1 2}\\
 &\leq \left( \int_{\mathbb R / \mathbb Z} \int_{-\frac 1 2}^{\frac 1 2}  \frac{ \int_{B_\varepsilon(0)} \left|\gamma(x+w -y) - \gamma(x-y) - (\gamma(x+w) - \gamma(x))\right|^2 \eta_\varepsilon(y) dy } {w^2}dw dx  \right)^ {\frac 1 2}
 \\ &= \left( \int_{B_\varepsilon(0)} \int_{\mathbb R / \mathbb Z} \int_{-\frac 1 2}^{\frac 1 2}  \frac{  \int_{B_\varepsilon(0)} \left|\gamma(x+w -y) - \gamma(x-y) - (\gamma(x+w) - \gamma(x))\right|^2 \eta_\varepsilon(y)  } {w^2}dw dx dy  \right)^ {\frac 1 2}
 \\ &\leq \left( \int_{\mathbb R / \mathbb Z} \int_{-\frac 1 2}^{\frac 1 2}  \frac{  \sup_{y \in B_\varepsilon(0)} \left|\gamma(x+w -y) - \gamma (x-y) - (\gamma(x+w) - \gamma(x))\right|^2 } {w^2}dw dx dy  \right)^ {\frac 1 2}.
\end{align*}
Clearly for all $\gamma \in C^\infty$ the above integral converges to $0$ for $\varepsilon \rightarrow 0$, as we can use Taylor's approximation twice to estimate it further by
\begin{align*}
 &\leq C \left( \int_{\mathbb R / \mathbb Z} \int_{-\frac 1 2}^{\frac 1 2}    \sup_{y \in B_\varepsilon(0)} \left| \int_0^1 \left( \gamma'(x-y+\tau w) - \gamma ' (x+\tau w)\right) d \tau \right|^2 dw dx dy  \right)^ {\frac 1 2} \\
 & \leq C \|\gamma''\|_{L^\infty} \varepsilon.
\end{align*}

For $\gamma \in W^{\frac 32,2}$ and $\delta >0$, we can find $\tilde \gamma$ with
$$
 \|\gamma- \tilde \gamma\|_{W^{\frac 32 ,2}} \leq \delta.
$$
Hence,
\begin{align*}
& \left( \int_{\mathbb R / \mathbb Z} \int_{-\frac 1 2}^{\frac 1 2}  \frac{  \sup_{y \in B_\varepsilon(0)} \left|\gamma(x+w -y) - \gamma (x-y) - (\gamma(x+w) - \gamma(x))\right|^2 } {w^2}dw dx dy  \right)^ {\frac 1 2} \\ &
\leq \left( \int_{\mathbb R / \mathbb Z} \int_{-\frac 1 2}^{\frac 1 2}  \frac{  \sup_{y \in B_\varepsilon(0)} \left|\tilde \gamma(x+w -y) - \tilde \gamma (x-y) - (\tilde \gamma(x+w) - \tilde \gamma(x))\right|^2 } {w^2}dw dx dy  \right)^ {\frac 1 2} 
\\ & + 2 \left( \int_{\mathbb R / \mathbb Z} \int_{-\frac 1 2}^{\frac 1 2}  \frac{  \sup_{y \in B_\varepsilon(0)} \left|\gamma(x+w) - \gamma (x) - (\tilde \gamma(x+w) - \tilde \gamma(x))\right|^2 } {w^2}dw dx dy  \right)^ {\frac 1 2}
\\ & = o(\varepsilon) + 2\delta
\end{align*}
for all $\delta >0$.
We conclude that
$$
 \tilde I_{\gamma_{\varepsilon}} \rightarrow \tilde I_\gamma
$$
in $L^1$.

This shows that the family of functions $I_{\gamma_\varepsilon}$ is uniformly integrable. Hence, Vitali's theorem implies
$E(\gamma_{\varepsilon}) \xrightarrow {\varepsilon \rightarrow 0} E(\gamma).$
\end{proof}

\subsubsection{Ishizeki's and Nagasawa's Decomposition of the M\"obius Energy. }
In \cite{Ishizeki2015} Ishizeki and Nagasawa found the decomposition
$$
 E_{\text{m\"ob}} (\gamma) = E^1 (\gamma) + E^2 (\gamma) + 4
$$
of the M\"obius energy where
$$
 E^1 (\gamma) := \iint_{(\mathbb R / \mathbb Z)^2} \frac {|\tau (x)- \tau(y)|^2}{2 |\gamma(x) - \gamma(y)|^2} |\gamma'(x_1) | |\gamma'(x_2)| dx_1 dx_2
$$
$\tau = \frac {\gamma'}{|\gamma'|}$ and
$$
 E^2 (\gamma) := \int_{(\mathbb R / \mathbb Z)^2} \frac 2 {|\gamma(x) - \gamma(y) |^2} \det
\begin{pmatrix}
 \tau (x) \cdot \tau (y) & (\gamma(x) - \gamma(y)) \cdot \tau (x)) \\
 (\gamma(x) - \gamma(y)) \cdot \tau (y)) & |\gamma(x) - \gamma(y)|^2
\end{pmatrix} |\gamma'(x)| |\gamma'(y)| dx dy.
$$
As in the proof of Lemma~\ref{lem:ConvOfMoebiusEnergy}, we can show
\begin{lemma} \label{lem:ConvOfE1E2}
 Let $\gamma \in C^{0,1}(\mathbb R / \mathbb Z, \mathbb R^n)$ be a curve of bounded M\"obius energy.
 Then $$ \lim_{\varepsilon \rightarrow 0} E^1(\gamma_\varepsilon)  = E^1 (\gamma) \quad \quad \text{ and } \quad \quad \quad \lim_{\varepsilon \rightarrow 0} E^2(\gamma_\varepsilon)  = E^2 (\gamma). $$
\end{lemma}

\begin{proof}
It is enough to show the convergence for $E^1$, as the statement for $E^2$ follows from the decomposition
$$ E_{\text{m\"ob}} = E^1 + E^2 + 4$$
by Ishizeki and Nagasawa \cite{Ishizeki2015}.
As $\gamma$ has bounded M\"obius energy, we know that $\gamma'\in VMO$. Theorem~\ref{thm:ApproximatingVMOCurves} shows that the integrand in the definition of $E^1$ converges pointwise.
From the bi-Lipschitz estimate we furthermore get
\begin{equation} \label{eq:BetterIntegrand}
  \frac {|\tau_\varepsilon (x)- \tau_\varepsilon(y)|^2}{2 |\gamma_\varepsilon(x) - \gamma_\varepsilon(y)|^2} |\gamma_\varepsilon'(x_1) | |\gamma_\varepsilon'(x_2)| \leq C \frac {| \gamma'_\varepsilon (x) - \gamma'_\varepsilon(y)|^2 }{|x-y|^2}.
\end{equation}
We have shown in the proof of Lemma~\ref{lem:ConvOfMoebiusEnergy} that the right-hand side in uniformly integrable --
and thus the integrands in the definition of $E^1$ are uniformly integrable and Vitali's theorem implies the assertion.
\end{proof}

\subsection{Proof of a Conjecture of Ishizeki and Nagasawa}

In \cite{Ishizeki2015}, Ishizeki and Nagasawa proved that for all curves $\gamma$ in $C^{1,1}$ we have $E^1 (\gamma) \geq 2 \pi^2$ and conjectured that the same is also true under the weaker but more natural condition $\gamma \in W^{\frac 32,2}$. Using the techniques we developed so far, we can now prove this conjecture quite easily.

\begin{theorem} [A conjecture of Nagasawa and Ishizeki] \label{thm:ConjectureNagasawa}
 We have $E^1 (\gamma) \geq 2 \pi^2$ for all regular curves $\gamma \in W^{\frac 3 2 ,2}(\mathbb R / \mathbb Z , \mathbb R^3)$.
\end{theorem}

\begin{proof}
 Let $\gamma_\varepsilon  = \gamma \ast \eta_\varepsilon$. Since 
 \begin{align*}
  E^1 (\gamma_\varepsilon) \rightarrow E^1 ( \gamma)
 \end{align*}
 and $E(\gamma_\varepsilon) \geq 2\pi^2$ as the inequality holds for smooth curves, we get 
 $ E^1(\gamma) \geq 2 \pi^2.$
\end{proof}

In the same paper, Ishizeki and Nagasawa also showed the M\"obius invariance of the energies $E^1$ and $E^2$ for curves of bounded M\"obius energy except for one important case: the case of an inversion on a sphere centered on the curve. For applications this seems to be one of the most important cases.  We can now prove also this last case -- and thus obtain full M\"obius invariance of the energies $E^1, E^2$ for curves of bounded M\"obius energy.

\begin{theorem} \label{thm:MoebiusInvariance}
 Let $\gamma \in C^{0,1}(\mathbb R / \mathbb Z)$ be a regular curve with bounded M\"obius energy and $I$ be an inversion on a sphere centered on $\gamma$. Then
 $$
  E^1(I\circ \gamma) = E^1 (\gamma)- 2 \pi^2\quad \text{ and } \quad  E^2(I \circ \gamma) = E^2(\gamma) + 2 \pi^2.
 $$ 
\end{theorem}

\begin{proof} We will show how to deduce this theorem form the M\"obius invariance for smooth curves and the invariance of the M\"obius energy. We only have to show the statement for $E^1$ as due to a theorem of Ishizeki and Nagasawa the sum
\begin{equation*}
  E^1 + E^2 
\end{equation*}
is known to be invariant under all M\"obius transformations \cite{Ishizeki2014}.

Let us assume that $\gamma$ is parametrized by arc length. We set $\gamma_\varepsilon := \gamma \ast \eta_\varepsilon$ and assume without loss of generality, that $0$ is the center of the inversion $I$. Then we can find $x_\varepsilon \rightarrow 0$ such that
 $0 \in \gamma_\varepsilon(\mathbb R / \mathbb Z) + x_\varepsilon.$
Let us denote by $\tilde \gamma_\varepsilon: \mathbb R \rightarrow \mathbb R^n$  a re-parameterization of $I \circ \gamma_\varepsilon$ by arc-length
such that $\tilde \gamma_\varepsilon (0)= (I \circ \gamma_{\varepsilon})(0)$ and let $\tilde \gamma: \mathbb R \rightarrow \mathbb R^n$  a re-parameterization of $I \circ \gamma$ by arc-length
such that $\tilde \gamma(0)= (I \circ \gamma)(0)$.
Then $\tilde \gamma_\varepsilon$ converges pointwise to $\tilde \gamma$.
 
The proof now relies on the following 
 
 \begin{claim} \label{claim}
  We have $$\lim_{\varepsilon \rightarrow 0}\lfloor \tilde \gamma'_\varepsilon - \tilde \gamma'\rfloor_{W^{\frac 1 2,2}(\mathbb R)} =0 ,$$
  where
  $$
   \lfloor \tilde \gamma'_\varepsilon - \tilde \gamma'\rfloor_{W^{\frac 1 2,2}(\mathbb R)}  := \left( \int_{\mathbb R} \int_{\mathbb R} \frac {|f(x) -f(y)|^2} {|x-y|^2} dx dy \right)^{\frac 1 2}
  $$
  denotes the Gagliardo semi-norm on $\mathbb R.$
 \end{claim}

 Let us prove the statement for $E^1$ in our theorem using this claim.
 On the one hand Lemma~\ref{lem:ConvOfE1E2} and the M\"obius invariance for smooth curves implies
 \begin{equation} \label{eq:Conv1}
   E^1(\tilde \gamma_\varepsilon)  + 2 \pi^2  = E^1(I \circ (\gamma_\varepsilon + x_\varepsilon))  + 2 \pi^2 = E^1 (\gamma_\varepsilon + x_\varepsilon) \xrightarrow{\varepsilon \downarrow 0} E^1(\gamma) .
 \end{equation}
 On the other hand, we use the estimate
\begin{equation*}
  \frac {|\tilde \tau_\varepsilon (x)- \tilde \tau_\varepsilon(y)|^2}{2 |\tilde \gamma_\varepsilon(x) - \tilde \gamma_\varepsilon(y)|^2} |\tilde \gamma_\varepsilon'(x_1) | |\tilde \gamma_\varepsilon'(x_2)| \leq C \frac {| \tilde \gamma'_\varepsilon (x) -  \tilde \gamma'_\varepsilon(y)|^2 }{|x-y|^2}
\end{equation*}
and follow the argument in the proof of Lemma~\ref{lem:ConvOfMoebiusEnergy} to see that the integrands in the definition of the energies $E^1(\tilde \gamma_\varepsilon)$ satisfy 
the assumptions of Vitali's theorem. Hence, 
\begin{equation} \label{eq:Conv2}
 \lim_{\varepsilon \rightarrow 0} E^{1} (\tilde \gamma_\varepsilon) = E^1 (\tilde \gamma).
\end{equation}
But \eqref{eq:Conv1} and \eqref{eq:Conv2} imply
\begin{equation*}
 E^1(\tilde \gamma) + 2 \pi^2 = E(\gamma).
\end{equation*}
\end{proof}

\begin{proof}[Proof of Claim~\ref{claim}] 
  We will show that the integrands appearing in the definition of 
  $$
   \lfloor \tilde \gamma_\varepsilon - \tilde \gamma \rfloor_{W^{\frac 1 2,2} (\mathbb R)} 
  $$
  converge pointwise to $0$ and are uniformly integrable. Then the claim follows from Vitali's theorem.  These integrands are
  $$
   \frac {|(\tilde \gamma_\varepsilon (x) - \tilde \gamma(x))-(\tilde \gamma_\varepsilon(y) - \tilde \gamma(y))|^2}{|x-y|^2}.
  $$

  As $\tilde \gamma_\varepsilon$ converges pointwise to $\tilde \gamma$, this integrands converge pointwise to $0$ for all $x \not=y$.
  
  Let us now first deal with the point $\infty$ and show that for every $\delta >0$ there is an $R>0$ such that
 \begin{equation} \label{eq:noConcentrationInfinity}
  \iint_{\mathbb R^2 \setminus (B_R)^2} \frac {|\tilde \gamma'_\varepsilon (x) - \tilde \gamma'_\varepsilon (y)|^2}{|x-y|^2} dx dy \leq \delta
 \end{equation}
 for all $\varepsilon >0$ small enough. 
 For this we use the M\"obius invariance of M\"obius energy \cite[Theorem~2.1]{Freedman1994}. Together with Fatou's lemma the latter implies
 $$
  E_{\text{m\"ob}}(\tilde \gamma) \leq \lim_{\varepsilon \rightarrow 0} E_{\text{m\"ob}} (\tilde \gamma_\varepsilon) = \lim_{\varepsilon \rightarrow 0} E_{\text{m\"ob}}( \gamma_\varepsilon) - 4 = E_{\text{m\"ob}}(\gamma)-4 = E(\tilde \gamma).
 $$
 Hence, 
 \begin{equation} \label{eq:noConcentration2}\lim_{\varepsilon \rightarrow 0} E _{\text{m\"ob}}(\tilde \gamma_\varepsilon ) = E_{\text{m\"ob}}(\tilde \gamma).\end{equation}
 For $\delta >0$ we now choose $R>0$ such that
 $$
  E_{B_R(0)}(\tilde \gamma) := \int_{B_R(0)}\int_{B_R(0)} \left( \frac 1 {|\tilde \gamma(x) - \tilde \gamma(y)|^2} - \frac 1 {|x-y|^2} \right) d x dy  \geq E(\tilde \gamma) - \delta.
 $$
 Then 
 $$
  E_{B_R(0)}(\tilde \gamma_\varepsilon) \geq E_{\text{m\"ob}}(\tilde \gamma) - 2 \delta
 $$
 for $\varepsilon >0$ small enough since else the lower semi-continuity of the M\"obius energy would imply
 $$
  E_{B_R(0)}(\tilde \gamma) \leq \liminf_{\varepsilon \rightarrow 0} E_{B_R(0)}(\tilde \gamma_\varepsilon) 
  \leq E_{\text{m\"ob}}(\tilde \gamma) - 2 \delta \leq E_{B_R(0)}(\tilde \gamma) - \delta  
 $$
 In view of \eqref{eq:noConcentration2} we even obtain
 \begin{equation*}
    E_{\text{m\"ob}}(\tilde \gamma) + 2 \delta \geq E_{\text{m\"ob}} (\tilde \gamma_\varepsilon) \geq E_{B_R(0)}(\tilde \gamma_\varepsilon) \geq E_{\text{m\"ob}}(\tilde \gamma) - 2 \delta
 \end{equation*}
 for all $\varepsilon >0$ sufficiently small and hence
 \begin{equation}
  \iint_{\mathbb R^2 \setminus (B_R(0))^2}  \left( \frac 1 {|\tilde \gamma_\varepsilon (x) - \tilde \gamma_\varepsilon(y)|^2} - \frac 1 {|x-y|^2} \right) dx dy = E_{\text{m\"ob}} (\tilde \gamma_\varepsilon) -  E_{B_R(0)}(\tilde \gamma_\varepsilon) \leq 2 \delta. 
 \end{equation}
 So the energy does not concentrate at the point infinity. Let us confer this into an statment for the Gagliardo semi-norm.
 
 One estimates
 \begin{align*}
  \frac 1 {|\tilde \gamma_\varepsilon (x) - \tilde \gamma_\varepsilon(y)|^2} - \frac 1 {|x-y|^2} & = \frac {|x-y|^2}{|\tilde \gamma_\varepsilon (x) - \tilde \gamma_\varepsilon (y)|^2 } \left(\frac {1-\frac {|\tilde \gamma_\varepsilon (x) - \tilde \gamma_\varepsilon(y)|^2}{|x-y|^2}} {|x-y|^2} \right)
  \\ & \geq \frac {1- \int_0^1 \int_0^1 \langle \tilde \gamma'_\varepsilon (x+t_1(y-x)), \tilde \gamma'_\varepsilon (x+t_2(y-x)) \rangle dt_1 dt_2 }{|x-y|^2} 
  \\ & = \frac 1 2 \frac { \int_0^1 \int_0^1 | \tilde \gamma'_\varepsilon (x+t_1(y-x)) - \tilde \gamma'_\varepsilon (x+t_2(y-x)) |^2 dt_1 dt_2 }{|x-y|^2}
 \end{align*}
 and hence
 \begin{equation} \label{eq:noConcentration3}
  \begin{aligned}
  \iint_{\mathbb R^2 \setminus (B_R(0))^2} &  \left( \frac 1 {|\tilde \gamma_\varepsilon (x) - \tilde \gamma_\varepsilon(y)|^2} - \frac 1 {|x-y|^2} \right) dx dy
  \\ &\geq \int_{\mathbb R \setminus B_R(0)} \int_{\mathbb R} \left( \frac 1 {|\tilde \gamma_\varepsilon (x) - \tilde \gamma_\varepsilon(x+w)|^2} - \frac 1 {|w|^2} \right) dw dx
  \\ & \geq \frac 1 2 \int_{\mathbb R \setminus B_R(0)} \int_{\mathbb R} \frac { \int_0^1 \int_0^1 | \tilde \gamma'_\varepsilon (x+t_1 w) - \tilde \gamma'_\varepsilon (x+t_2 w) |^2 dt_1 dt_2 }{|w|^2} dw dx
  \\ & \geq \frac c2 \iint_{\mathbb R^2 \setminus (B_R)^2} \frac {|\tilde \gamma'_\varepsilon (x) - \tilde \gamma'_\varepsilon (y)|^2}{|x-y|^2} dx dy.
 \end{aligned}
 \end{equation}
 Estimates \eqref{eq:noConcentration3} and \eqref{eq:noConcentration2} imply \eqref{eq:noConcentrationInfinity}.

 We will now deduce that 
 \begin{equation} \label{eq:localConvergence}
   \lim_{\varepsilon \rightarrow0 }\int_{B_R} \int_{B_R} \frac{|(\tilde \gamma_\varepsilon'(x) - \tilde \gamma'(x)) - (\tilde \gamma_\varepsilon'(y) - \tilde \gamma'(y))|^2 } {|x-y|^2} dx dy =0
 \end{equation}
 again using Vitali's theorem. As noted before, we know that the integrand converges pointwise almost everywhere to $0$.
 
 To show uniform integrability of the integrands we use the estimate
 \begin{equation*}
  \frac{|\tilde \gamma_\varepsilon'(x) - \tilde \gamma'(x)) - (\tilde \gamma_\varepsilon'(y) - \tilde \gamma'(y)|^2 } {|x-y|^2}
  \leq \frac{|\tilde \gamma_\varepsilon'(x)  - \tilde \gamma_\varepsilon'(y)|^2 } {|x-y|^2} + \frac{|\tilde \gamma'(x)  - \tilde \gamma'(y)|^2 } {|x-y|^2}.
 \end{equation*}
Of course, one only has to show that the first summand in uniformly integrable.  
 Let $s(x)$ be the re-parametrization that satisfies $\tilde \gamma_\varepsilon (x) = I \circ \gamma_\varepsilon \circ s$ and $\delta >0$ be such that $(-R,R) \subset (I \circ \gamma_{\varepsilon}) ((\mathbb R / \mathbb Z) \setminus B_{\delta (0)})$ and $\psi(x,y) := (s(x), s(y)).$ As in the proof of Theorem~\ref{thm:ApproximationArcLength} we get for $E \subset (-R,R) ^2$
\begin{align*}
 \iint_{E} \frac {|\tilde \gamma'_\varepsilon(x) - \tilde \gamma_\varepsilon'(y)|^2}{|x-y|^2} dx dy & \leq C \iint_{\psi^{-1}(E)} \frac {| \gamma'_\varepsilon(x) - \gamma_\varepsilon'(y)|^2}{|x-y|^2} |\gamma_\varepsilon'(x)| |\gamma_\varepsilon'(y) |dx dy \\
 & \leq C \iint_{\psi^{-1}(E)} \frac {| \gamma'_\varepsilon(x) -  \gamma'_\varepsilon(y)|^2}{|x-y|^2}|dx dy.
\end{align*}
Since the integrands
\begin{equation*}
  \frac {| \gamma'_\varepsilon(x) -  \gamma_\varepsilon'(y)|^2}{|x-y|^2}
\end{equation*}
are uniformly integrable for every $\varepsilon_0>0$ there is an $\delta >0$ such that
$|\psi^{-1}(E)|\leq \delta$ implies
\begin{equation*}
 \iint_{E} \frac {|\tilde \gamma'_\varepsilon(x) - \tilde \gamma_\varepsilon'(y)|^2}{|x-y|^2} dx dy \leq \varepsilon.
\end{equation*}
But as $\psi^{-1}$ is a Lipschitz mapping, we get that there is an $\tilde \delta >0$ such that $|E| \leq \tilde \delta$ implies $|\psi^{-1}(E)|$ and hence
\begin{equation*}
 \iint_{E} \frac {|\tilde \gamma'_\varepsilon(x) - \tilde \gamma_\varepsilon'(y)|^2}{|x-y|^2} dx dy \leq \varepsilon.
\end{equation*}
Hence, Vitali's theorem implies that
  \begin{equation*}
 \lim_{\varepsilon \rightarrow 0} \int_{B_R} \int_{B_R} \frac{|(\gamma_\varepsilon'(x) - \tilde \gamma'(x)) - (\gamma_\varepsilon'(y) - \tilde \gamma'(y))|^2 } {|x-y|^2} dx dy =0
 \end{equation*}
 
 Let us now conclude the proof of the claim. For $\delta>0$ we first use \eqref{eq:noConcentrationInfinity} to get an $R>0$ such that 
  \begin{align*}
  \iint_{\mathbb R^2 \setminus (B_R)^2} \frac {|\tilde \gamma'_\varepsilon (x) - \tilde \gamma'_\varepsilon (y)|^2}{|x-y|^2} dx dy \leq \delta
 \end{align*}
 for all $\varepsilon >0$ small enough. 
 Then \eqref{eq:localConvergence} implies
 \begin{align*}
  \limsup_{\varepsilon \rightarrow 0} \lfloor \tilde \gamma'_\varepsilon - \tilde \gamma'\rfloor^2_{W^{\frac 1 2,2}} 
  & = \limsup_{\varepsilon \rightarrow 0} \iint_{\mathbb R^2 \setminus (B_R)^2} \frac {|\tilde \gamma'_\varepsilon (x) - \tilde \gamma'_\varepsilon (y)|^2}{|x-y|^2} dx dy
  + \iint_{B_{R^2}} \frac {|\tilde \gamma'_\varepsilon (x) - \tilde \gamma'_\varepsilon (y)|^2}{|x-y|^2} dx dy  \\ & \leq \delta.
 \end{align*}
 and thus
 \begin{equation*}
  \lim_{\varepsilon \rightarrow 0} \lfloor \tilde \gamma'_\varepsilon - \tilde \gamma'\rfloor_{W^{\frac 1 2,2}}  =0.
 \end{equation*}
\end{proof}

With the help of Theorem\ref{thm:MoebiusInvariance}, we can now also discuss the case of equality in Theorem~\ref{thm:ConjectureNagasawa} to get the following extension of Corollary~4.1 in \cite{Ishizeki2015}.

\begin{theorem} \label{thm:equality}
  We have $E^1 (\gamma) \geq 2 \pi^2$ for all regular curves $\gamma \in W^{\frac 3 2 ,2}(\mathbb R / \mathbb Z , \mathbb R^3)$ with equality if and only if $\gamma$ is a circle.
\end{theorem}

We omit the proof of Theorem~\ref{thm:equality} as it is literally the same as the proof of Corollary~4.1 in \cite{Ishizeki2015}  where one only uses Theorem~\ref{thm:MoebiusInvariance} instead of Theorem~1.2 in \cite{Ishizeki2015}.

\subsection{$\Gamma$-convergence of the Discrete M\"obius Energies of Scholtes}

Let us extend the $\Gamma$-convergence result of Scholtes in \cite{Scholtes2014}. Scholtes introduced the discretized M\"obius energy
\begin{equation*}
 E_n (p) = \sum_{i,j =1}^m \left(\frac 1 {|p(a_i) - p(a_j)|^2}  - \frac 1 {d_p(a_i,a_j)^2}\right) d_p(a_{i+1},a_i) d_p(a_{j+1},a_j).
\end{equation*}
of a polygon $p:\mathbb R/ \mathbb Z \rightarrow \mathbb R^n$ with vertices $p(a_i)$, $a_i \in [0,1)$, $i=1,\ldots, m$.

\begin{theorem} [$\Gamma$-convergence of discrete M\"obius energies]
 Let $q \in [1,\infty)$. We have
 $$
  E_n \xrightarrow{\Gamma} E_{\text{m\"ob}} 
 $$
 on the space of curves $C^{0,1}(\mathbb R / \mathbb Z, \mathbb R^n)$ of unit velocity equipped with the $L^q$ and $W^{1,q}$-norm.
\end{theorem}

Scholtes proved this theorem for curves which are in $C^1$ -- which again is not implied by bounded M\"obius energy.

\begin{proof}
 Since the $\liminf$-inequality was already shown by Scholtes, we only have to proof the $\limsup$ inequality. Scholtes has already shown that the $\limsup$ inequality holds for $C^1$ curves. If now $\gamma$ is a regular curves with bounded M\"obius energy, we can consider the smoothened curves $\gamma_\varepsilon = \gamma \ast \eta_\varepsilon$. By Lemma~\ref{lem:ConvOfMoebiusEnergy} we have
 $\lim_{m \rightarrow \infty} E_{\text{m\"ob}}(\gamma_{\frac 1 m}) = E_{\text{m\"ob}}(\gamma).$ By the $\limsup$-inequality of Scholtes, we can find $\tilde \gamma_m$ with $E_m(\gamma_\frac 1 m) \geq E(\gamma_{\frac 1 m}) - \frac 1 m$. Hence,
 $$
  \limsup_{m\rightarrow \infty} E_m(\gamma_{\frac 1m}) \geq \lim_{m \rightarrow \infty} \left( E_{\text{m\"ob}}(\gamma_{\frac 1 m}) - \frac 1 m \right) = E_{\text{m\"ob}}(\gamma).
 $$
 We now let $ \tilde \gamma_m$ denote the re-parametrization of the curve $\frac 1 {L(\gamma_{\frac 1m})} \gamma_{\frac 1m}$ by arc-length that does not change the point $0$.
 Then by Theorem~\ref{thm:ApproximationArcLength}, 
 $$
   \tilde \gamma_m \rightarrow \gamma \text{ in } W^{\frac 3 2,2} (\mathbb R / \mathbb Z, \mathbb R^n)
 $$
 and still
 $$
  \limsup_{m\rightarrow \infty} E_m(\tilde \gamma_{m}) = \limsup_{m\rightarrow \infty} E_m(\gamma_{\frac 1m}) \geq \lim_{m \rightarrow \infty} \left( E_{\text{m\"ob}}(\gamma_{\frac 1 m}) - \frac 1 m \right) = E_{\text{m\"ob}}(\gamma).
 $$
\end{proof}

\subsection{Inscribing Equilateral polygons}

With the tools we have at hand, we can also extend a result of Wu \cite{Wu2004} on inscribed regular polygons in the following way
\begin{theorem}
 Let $\gamma \in C^{0,1}(\mathbb R / \mathbb Z, \mathbb R^d)$ be a regular curve with $\gamma' \in VMO$. Then for every $n\in \mathbb N$, $n \geq 2$ and any $x_0 \in \mathbb R / \mathbb Z$ there is an inscribed $n$-gon starting with the point $\gamma(x_0).$  
\end{theorem}

The proof is based on the fact, that there is a lower bound $c_n >0$ of the Gromov distortion of equilateral $n$-gons as the infimum of the Gromov distortion is attained for an equilateral $n$-gon and thus cannot be $0$.

\begin{proof}
 Let $\gamma_\varepsilon = \eta_\varepsilon \ast \gamma$ be the standard mollified curves and $p_k$ the inscribed, equilateral $n$-gon with point $\gamma_{\frac 1k} (x_0)$ for $k$ so large that $\gamma_{\frac 1k}$ is a regular curve. We first note that $\inf_{k \in \mathbb N} \diam p_k =0$ would imply 
 \begin{equation} \label{eq:LowerBoundDistortion}
  \sup_{x\not= y \in \mathbb R / \mathbb Z, |x-y| \leq \diam p_k } \frac {|\gamma_{\frac 1k }(x) - \gamma_{\frac 1k }(y)|} {d_{\gamma_{ \frac 1 k}}(x,y)}\geq c_n >0
 \end{equation}
 where $c_n$ is a lower bound of the Gromov distortion of a equilateral $n$-gon. 
 
 If we set $$VMO(r) =\sup_{x \in \mathbb R / \mathbb Z} \frac 1 {2r} \int_{B_r(x)} |\gamma'(y) - \overline{\gamma'}_{B_r(x)}| dy,$$ then on the other hand we have
  \begin{align*}
 \frac 1 {2r}  \int_{B_r(x)}| \gamma'_\varepsilon - \overline{\left(\gamma'_\varepsilon \right)}_{B_r(x)}  | dy 
 &= \frac 1 {2r} \int_{B_r(x)} \left|  \int_ {B_\varepsilon(0)} \left(\gamma'(y+z) - \overline{\left(\gamma'_\varepsilon \right)}_{B_r(x+z)}\right) \eta_\varepsilon(z) dz  \right| dy \\
 &\leq VMO(r) \rightarrow 0
 \end{align*} 
 as $r\rightarrow 0$ since $\gamma'$ belongs to $VMO$. Hence, 
 $$
  \sup_{x\not= y \in \mathbb R / \mathbb Z, |x-y| \leq r } \frac {|\gamma_{\frac 1 n}(x) - \gamma_{\frac 1n}(y)|}{d_\gamma(x,y)} \rightarrow 0
 $$
 as $r \rightarrow 0$. This contradicts inequality \eqref{eq:LowerBoundDistortion}.
\end{proof}

\end{document}